\documentclass[11pt,a4paper]{article}

\usepackage{amsmath}
\usepackage{amsthm}
\usepackage{amssymb}
\usepackage{amscd}
\usepackage[all]{xy}

\title{Strongly Liftable Schemes and the Kawamata-Viehweg Vanishing 
in Positive Characteristic II
\footnote{This paper was partially supported by the National Natural Science 
Foundation of China (Grant No.\ 10901037) and the Ph.D.\ Programs Foundation 
of Ministry of Education of China (Grant No.\ 20090071120004).}}
\author{Qihong Xie}
\date{}
\theoremstyle{plain}
\newtheorem{prop}{Proposition}[section]
\newtheorem{lem}[prop]{Lemma}
\newtheorem{thm}[prop]{Theorem}
\newtheorem{cor}[prop]{Corollary}

\theoremstyle{definition}
\newtheorem{defn}[prop]{Definition}
\newtheorem*{ack}{Acknowledgments}
\newtheorem*{nota}{Notation}

\theoremstyle{remark}
\newtheorem{rem}[prop]{Remark}
\newtheorem{ex}[prop]{Example}

\newcommand{\Q}{\mathbb Q}
\newcommand{\R}{\mathbb R}
\newcommand{\C}{\mathbb C}
\newcommand{\Z}{\mathbb Z}

\newcommand{\F}{\mathbb F}
\newcommand{\A}{\mathbb A}
\newcommand{\PP}{\mathbb P}
\newcommand{\OO}{\mathcal O}
\newcommand{\II}{\mathcal I}
\newcommand{\EE}{\mathcal E}

\newcommand{\HH}{\mathcal H}
\newcommand{\LL}{\mathcal L}
\newcommand{\GG}{\mathcal G}

\newcommand{\FF}{\mathcal F}

\newcommand{\QQ}{\mathcal Q}

\newcommand{\cA}{\mathcal A}

\newcommand{\Pic}{\mathop{\rm Pic}\nolimits}
\newcommand{\Div}{\mathop{\rm Div}\nolimits}

\newcommand{\Supp}{\mathop{\rm Supp}\nolimits}
\newcommand{\Sing}{\mathop{\rm Sing}\nolimits}

\newcommand{\ch}{\mathop{\rm char}\nolimits}

\newcommand{\Image}{\mathop{\rm Im}\nolimits}
\newcommand{\Spec}{\mathop{\bf Spec}\nolimits}
\newcommand{\spec}{\mathop{\rm Spec}\nolimits}

\newcommand{\Trace}{\mathop{\rm Trace}\nolimits}

\newcommand{\divisor}{\mathop{\rm div}\nolimits}

\newcommand{\ra}{\rightarrow}

\newcommand{\wt}{\widetilde}
\newcommand{\la}{\lambda}

\begin{document}

\maketitle

\begin{abstract}
A smooth scheme $X$ over a field $k$ of positive characteristic 
is said to be strongly liftable, if $X$ and all prime divisors 
on $X$ can be lifted simultaneously over $W_2(k)$. 
In this paper, first we prove that smooth toric varieties are strongly 
liftable. As a corollary, we obtain the Kawamata-Viehweg vanishing theorem 
for smooth projective toric varieties. Second, we prove the Kawamata-Viehweg 
vanishing theorem for normal projective surfaces which are birational to a 
strongly liftable smooth projective surface. Finally, we deduce the cyclic 
cover trick over $W_2(k)$, which can be used to construct a large class of 
liftable smooth projective varieties.
\end{abstract}

\setcounter{section}{0}
\section{Introduction}\label{S1}

Throughout this paper, we always work over {\it an algebraically 
closed field $k$ of characteristic $p>0$} unless otherwise stated. 
A smooth scheme $X$ is said to be strongly liftable, if $X$ and 
all prime divisors on $X$ can be lifted simultaneously over $W_2(k)$. 
This notion was first introduced in \cite{xie10} to study the 
Kawamata-Viehweg vanishing theorem in positive characteristic, 
furthermore, some examples and properties of strongly liftable schemes 
were also given in \cite{xie10}.

In this paper, we shall continue to study strongly liftable schemes. 
First of all, we find an important class of strongly liftable schemes 
with simple structures.

\begin{thm}\label{1.1}
Smooth toric varieties are strongly liftable over $W_2(k)$.
\end{thm}

As a consequence, we obtain the Kawamata-Viehweg vanishing theorem 
on smooth projective toric varieties for ample $\Q$-divisors which are 
not necessarily torus invariant.

\begin{cor}\label{1.2}
Let $X$ be a smooth projective toric variety of dimension $d$, 
$H$ an ample $\Q$-divisor on $X$, and $D$ a simple normal crossing 
divisor containing $\Supp(\langle H\rangle)$. Then 
\[ H^i(X,\Omega_X^j(\log D)(-\ulcorner H\urcorner))=0 \,\,\,\,
\hbox{holds for any} \,\,\,\, i+j<\inf(d,p). \]
In particular, $H^i(X,K_X+\ulcorner H\urcorner)=0$ holds 
for any $i>d-\inf(d,p)$.
\end{cor}

Second, we generalize \cite[Theorem 1.4]{xie10} slightly to the case 
where no singularity assumption is made.

\begin{thm}\label{1.3}
Let $X$ be a normal projective surface and $H$ a nef and big $\Q$-divisor 
on $X$. If $X$ is birational to a strongly liftable smooth projective 
surface $Z$, then $H^1(X,K_X+\ulcorner H\urcorner)=0$ holds.
\end{thm}

As a corollary, we obtain the Kawamata-Viehweg vanishing theorem for 
rational surfaces, which is a generalization of \cite[Theorem 1.4]{xie09}.

\begin{cor}\label{1.4}
Let $X$ be a normal projective rational surface and $H$ a nef and big 
$\Q$-divisor on $X$. Then $H^1(X,K_X+\ulcorner H\urcorner)=0$ holds.
\end{cor}

Finally, we deduce an explicit statement of the cyclic cover trick over 
$W_2(k)$ (see Theorem \ref{4.1} for more details). By means of the cyclic 
cover trick, we can construct a large class of liftable smooth projective 
varieties from certain strongly liftable varieties with simple structures, 
e.g.\ toric varieties. Namely, we have the following corollary.

\begin{cor}\label{1.5}
Let $X$ be a smooth projective toric variety, and $\LL$ an invertible sheaf 
on $X$. Let $N$ be a positive integer prime to $p$, and $D$ an effective 
divisor on $X$ with $\LL^N=\OO_X(D)$ and $\Sing(D_{\rm red})=\emptyset$. 
Let $\pi:Y\ra X$ be the cyclic cover obtained by taking the $N$-th root 
out of $D$. Then $Y$ is a liftable smooth projective scheme.
\end{cor}

Especially, it follows from Corollary \ref{1.5} that there do exist many 
liftable smooth projective varieties of general type, whose existence is 
helpful for studying birational geometry of algebraic varieties in positive 
characteristic. In general, the schemes obtained by taking cyclic covers 
over strongly liftable schemes are no longer strongly liftable (see Remark 
\ref{4.6} for more details), which shows that the class of strongly liftable 
schemes is really restrictive.

In \S \ref{S2}, we will recall some definitions and preliminary results 
of liftings over $W_2(k)$. \S \ref{S3} is devoted to the proofs of the main 
theorems. The cyclic cover trick over $W_2(k)$ will be treated in \S \ref{S4}. 
In \S \ref{S5}, we will give some corrections to the mistakes in \cite{xie10}. 
For the necessary notions and results in birational geometry, 
we refer the reader to \cite{kmm} and \cite{km}.

\begin{nota}
We use $\sim$ to denote linear equivalence, 
$\equiv$ to denote numerical equivalence, 
and $[B]=\sum [b_i] B_i$ (resp.\ 
$\ulcorner B\urcorner=\sum \ulcorner b_i\urcorner B_i$, 
$\langle B\rangle=\sum \langle b_i\rangle B_i$, 
$\{B\}=\sum\{b_i\}B_i$) to denote the round-down (resp.\ 
round-up, fractional part, upper fractional part) 
of a $\Q$-divisor $B=\sum b_iB_i$, where for a real number $b$, 
$[b]:=\max\{ n\in\Z \,|\,n\leq b \}$, $\ulcorner b\urcorner:=-[-b]$, 
$\langle b\rangle:=b-[b]$ and $\{ b\}:=\ulcorner b\urcorner-b$. 
We use $\Sing(D_{\rm red})$ to denote the singular locus of the 
reduced part of a divisor $D$.
\end{nota}

\begin{ack}
I would like to express my gratitude to Professor Luc Illusie for 
pointing out some errors in \cite{xie10} and an earlier version of 
this paper. I would also like to thank Professors Osamu Fujino and 
Fumio Sakai for useful comments.
\end{ack}

\section{Preliminaries}\label{S2}

\begin{defn}\label{2.1}
Let $W_2(k)$ be the ring of Witt vectors of length two of $k$. 
Then $W_2(k)$ is flat over $\Z/p^2\Z$, and $W_2(k)\otimes_{\Z/p^2\Z}\F_p=k$. 
For the explicit construction and further properties of $W_2(k)$, 
we refer the reader to \cite[II.6]{se62}. The following definition 
\cite[Definition 8.11]{ev} generalizes the definition \cite[1.6]{di} 
of liftings of $k$-schemes over $W_2(k)$.

Let $X$ be a noetherian scheme over $k$, and $D=\sum D_i$ a reduced Cartier 
divisor on $X$. A lifting of $(X,D)$ over $W_2(k)$ consists of a scheme 
$\wt{X}$ and closed subschemes $\wt{D}_i\subset\wt{X}$, all defined and 
flat over $W_2(k)$ such that $X=\wt{X}\times_{\spec W_2(k)}\spec k$ and 
$D_i=\wt{D}_i\times_{\spec W_2(k)}\spec k$. We write 
$\wt{D}=\sum \wt{D}_i$ and say that $(\wt{X},\wt{D})$ is a lifting 
of $(X,D)$ over $W_2(k)$, if no confusion is likely.

Let $\LL$ be an invertible sheaf on $X$. A lifting of $(X,\LL)$ consists 
of a lifting $\wt{X}$ of $X$ over $W_2(k)$ and an invertible sheaf $\wt{\LL}$ 
on $\wt{X}$ such that $\wt{\LL}|_X=\LL$. For simplicity, we say that 
$\wt{\LL}$ is a lifting of $\LL$ on $\wt{X}$, if no confusion is likely.
\end{defn}

Let $\wt{X}$ be a lifting of $X$ over $W_2(k)$. Then $\OO_{\wt{X}}$ is flat 
over $W_2(k)$, hence flat over $\Z/p^2\Z$. Note that there are an exact 
sequence of $\Z/p^2\Z$-modules:
\[
0\ra p\cdot\Z/p^2\Z\ra \Z/p^2\Z\stackrel{r}{\ra} \Z/p\Z\ra 0
\]
and a $\Z/p^2\Z$-module isomorphism $p:\Z/p\Z\ra p\cdot\Z/p^2\Z$. 
Tensoring the above by $\OO_{\wt{X}}$, we obtain an exact sequence of 
$\OO_{\wt{X}}$-modules:
\begin{eqnarray}
0\ra p\cdot\OO_{\wt{X}}\ra \OO_{\wt{X}}\stackrel{r}{\ra} 
\OO_X\ra 0, \label{es1}
\end{eqnarray}
and an $\OO_{\wt{X}}$-module isomorphism
\begin{eqnarray}
p:\OO_X\ra p\cdot\OO_{\wt{X}}, \label{es2}
\end{eqnarray}
where $r$ is the reduction modulo $p$ satisfying $p(x)=p\wt{x}$, 
$r(\wt{x})=x$ for $x\in\OO_X$, $\wt{x}\in\OO_{\wt{X}}$.

The following lemma has already been proved in \cite[Lemmas 8.13 and 8.14]{ev}.

\begin{lem}\label{2.2}
Let $(\wt{X},\wt{D})$ be a lifting of $(X,D)$ as in Definition \ref{2.1}. 
If $X$ is smooth over $k$ and $D\subset X$ is simple normal crossing, 
then $\wt{X}$ is smooth over $W_2(k)$ and $\wt{D}\subset\wt{X}$ is 
relatively simple normal crossing over $W_2(k)$.
\end{lem}

\begin{proof}
Since the statement is local, we may assume that there is an \'etale 
morphism $\varphi:X\ra\A^n_k=\spec k[t_1,\cdots,t_n]$, such that 
$\varphi^*(t_i)$ give a regular system of parameters $(x_1,\cdots,x_n)$ on 
$X$, and $D=\sum D_i\subset X$ is defined by the equation $x_1\cdots x_r=0$ 
for some $r\leq n$. Take $\wt{x}_i\in\OO_{\wt{X}}$ with $r(\wt{x}_i)=x_i$ 
($1\leq i\leq n$), and define
\[
\wt{\varphi}^*:W_2(k)[t_1,\cdots,t_n]\ra \OO_{\wt{X}}
\]
as a $W_2(k)$-algebra homomorphism by $\wt{\varphi}^*(t_i)=\wt{x}_i$ 
($1\leq i\leq n$), which gives rise to a morphism $\wt{\varphi}:\wt{X}
\ra\A^n_{W_2(k)}$. First of all, we shall prove that $\wt{\varphi}$ is 
an \'etale morphism, which implies that $\wt{X}$ is smooth over $W_2(k)$.

Without loss of generality, we may assume that $\OO_X$ is a free 
$\OO_{\A^n_k}$-module with generators $g_1,\cdots,g_m$. Take $\wt{g}_i
\in\OO_{\wt{X}}$ lifting $g_i$ ($1\leq i\leq m$), then for any $\wt{x}
\in\OO_{\wt{X}}$, we can write $x=r(\wt{x})=\sum^m_{i=1}\la_ig_i$ 
for some $\la_i\in\OO_{\A^n_k}$. Take $\wt{\la}_i\in\OO_{\A^n_{W_2(k)}}$ 
lifting $\la_i$ ($1\leq i\leq m$), then by the exact sequence (\ref{es1}) 
we have $\wt{x}-\sum^m_{i=1}\wt{\la}_i\wt{g}_i\in p\cdot\OO_{\wt{X}}$. 
Thus we can find $\wt{\mu}_i\in\OO_{\A^n_{W_2(k)}}$ ($1\leq i\leq m$), 
such that $\mu_i=r(\wt{\mu}_i)$ and
\[
\wt{x}-\sum^m_{i=1}\wt{\la}_i\wt{g}_i=p(\sum^m_{i=1}\mu_ig_i)
=\sum^m_{i=1}p\wt{\mu}_i\wt{g}_i,
\]
hence we have $\wt{x}=\sum^m_{i=1}(\wt{\la}_i+p\wt{\mu}_i)\wt{g}_i$. 
Assume $\sum^m_{i=1}\wt{\la}_i\wt{g}_i=0$ for some $\wt{\la}_i\in\OO
_{\A^n_{W_2(k)}}$. Then $\la_i=r(\wt{\la}_i)$ satisfy $\sum^m_{i=1}
\la_ig_i=0$, which implies $\la_i=0$ ($1\leq i\leq m$). Thus we can 
find $\wt{\mu}_i\in\OO_{\A^n_{W_2(k)}}$ ($1\leq i\leq m$), such that 
$\mu_i=r(\wt{\mu}_i)$ and $\wt{\la}_i=p(\mu_i)=p\wt{\mu}_i$, hence
\[
p(\sum^m_{i=1}\mu_ig_i)=\sum^m_{i=1}p\wt{\mu}_i\wt{g}_i
=\sum^m_{i=1}\wt{\la}_i\wt{g}_i=0.
\]
Since $p$ is an isomorphism, we have $\sum^m_{i=1}\mu_ig_i=0$, hence 
$\mu_i=0$ and $\wt{\la}_i=0$ ($1\leq i\leq m$). Therefore $\OO_{\wt{X}}$ 
is a free $\OO_{\A^n_{W_2(k)}}$-module and $\wt{\varphi}$ is flat.

Since $\varphi:X\ra\A^n_k$ is \'etale, we have $\Omega^1_{X/k}=\varphi^*
\Omega^1_{\A^n_k/k}=\bigoplus^n_{i=1}\OO_Xdx_i$. By definition, we have 
$\wt{\varphi}^*\Omega^1_{\A^n_{W_2(k)}/W_2(k)}=\bigoplus^n_{i=1}\OO_{\wt{X}}
d\wt{x}_i$. Combining the exact sequence (\ref{es1}) and the isomorphism 
(\ref{es2}), we obtain an exact sequence of $\OO_{\wt{X}}$-modules:
\begin{eqnarray}
0\ra \OO_X\stackrel{p}{\ra} \OO_{\wt{X}}\stackrel{r}{\ra} \OO_X\ra 0. 
\label{es3}
\end{eqnarray}
Tensoring the exact sequence (\ref{es3}) by $\Omega^1_{\wt{X}/W_2(k)}$, 
we obtain an exact sequence:
\[
\Omega^1_{X/k}\stackrel{p}{\ra} \Omega^1_{\wt{X}/W_2(k)}\stackrel{r}{\ra} 
\Omega^1_{X/k}\ra 0.
\]
We shall prove that the natural homomorphism $\wt{\varphi}^*\Omega^1
_{\A^n_{W_2(k)}/W_2(k)}\ra \Omega^1_{\wt{X}/W_2(k)}$ is surjective. 
For any $\wt{\omega}\in\Omega^1_{\wt{X}/W_2(k)}$, we can write 
$\omega=r(\wt{\omega})=\sum^n_{i=1}\la_idx_i\in\Omega^1_{X/k}$ for 
some $\la_i\in\OO_X$. Take $\wt{\la}_i\in\OO_{\wt{X}}$ lifting $\la_i$ 
($1\leq i\leq n$), then $\wt{\omega}-\sum^n_{i=1}\wt{\la}_id\wt{x}_i
\in\Image(\Omega^1_{X/k}\stackrel{p}{\ra} \Omega^1_{\wt{X}/W_2(k)})$. 
Thus we can find $\wt{\mu}_i\in\OO_{\wt{X}}$ ($1\leq i\leq n$) such that
$\mu_i=r(\wt{\mu}_i)$ and
\[
\wt{\omega}-\sum^n_{i=1}\wt{\la}_id\wt{x}_i=p(\sum^n_{i=1}\mu_idx_i)
=\sum^n_{i=1}p\wt{\mu}_id\wt{x}_i,
\]
hence $\wt{\omega}=\sum^n_{i=1}(\wt{\la}_i+p\wt{\mu}_i)d\wt{x}_i$, so 
$\wt{\varphi}^*\Omega^1_{\A^n_{W_2(k)}/W_2(k)}\ra \Omega^1_{\wt{X}/W_2(k)}$ 
is surjective. By the first exact sequence associated to $\wt{\varphi}:
\wt{X}\ra\A^n_{W_2(k)}$ \cite[Proposition II.8.11]{ha77}, we have 
$\Omega^1_{\wt{X}/\A^n_{W_2(k)}}=0$, which together with the flatness 
of $\wt{\varphi}$ implies that $\wt{\varphi}$ is \'etale 
(see \cite[Exercise III.10.3]{ha77}), so $\wt{X}$ is smooth over $W_2(k)$.

By assumption, $D_i$ is defined by the equation $x_i=0$ ($1\leq i\leq r$). 
Since $\wt{D}_i$ is a lifting of $D_i$ over $W_2(k)$, the flatness of 
$\OO_{\wt{X}}$ and $\OO_{\wt{D}_i}$ over $W_2(k)$ implies that the ideal 
sheaf $\II_{\wt{D}_i}$ of $\wt{D}_i$ is flat over $W_2(k)$. Thus we have 
an exact sequence:
\[
0\ra p\cdot\II_{\wt{D}_i}\ra \II_{\wt{D}_i}\stackrel{r}{\ra} \II_{D_i}
\ra 0,
\]
and an isomorphism $p:\II_{D_i}\ra p\cdot\II_{\wt{D}_i}$, where $\II_{D_i}$ 
is the ideal sheaf of $D_i$. For any $\wt{g}\in\II_{\wt{D}_i}$, we can write 
$g=r(\wt{g})=\la_ix_i$ for some $\la_i\in\OO_X$. Take $\wt{\la}_i\in
\OO_{\wt{X}}$ lifting $\la_i$ ($1\leq i\leq r$), then $\wt{g}-\wt{\la}_i
\wt{x}_i\in p\cdot\II_{\wt{D}_i}$. Thus we can find $\wt{\mu}_i\in
\OO_{\wt{X}}$ ($1\leq i\leq r$) such that $\mu_i=r(\wt{\mu}_i)$ and
\[
\wt{g}-\wt{\la}_i\wt{x}_i=p(\mu_ix_i)=p\wt{\mu}_i\wt{x}_i,
\]
hence $\wt{g}=(\wt{\la}_i+p\wt{\mu}_i)\wt{x}_i$, so $\II_{\wt{D}_i}$ is 
generated by $\wt{x}_i$, i.e.\ $\wt{x}_i=0$ is a defining equation for 
$\wt{D}_i$  ($1\leq i\leq r$). Thus $\wt{D}=\sum\wt{D}_i\subset\wt{X}$ 
is relatively simple normal crossing over $W_2(k)$.
\end{proof}

We can use Lemma \ref{2.2} to deduce Bertini's theorem for ample 
invertible sheaves on smooth projective schemes over $W_2(k)$.

\begin{thm}\label{2.3}
Let $\wt{X}$ be a smooth projective scheme over $W_2(k)$, and $\wt{\LL}$ 
an ample invertible sheaf on $\wt{X}$. Then there is a positive integer 
$m$ such that $\wt{\LL}^m$ is very ample and associated to a general 
section $\wt{s}\in H^0(\wt{X},\wt{\LL}^m)$, the divisor of zeros 
$\wt{D}=\divisor_0(\wt{s})$ is smooth over $W_2(k)$.
\end{thm}

\begin{proof}
The natural surjection $W_2(k)\stackrel{r}{\ra} k$ induces the closed 
immersion $\iota:\spec k\hookrightarrow \spec W_2(k)$. Let $X=\wt{X}\times_
{\spec W_2(k)}\spec k$, $\iota:X\hookrightarrow\wt{X}$ the closed immersion, 
and $\LL=\iota^*\wt{\LL}=\wt{\LL}|_X$ the induced invertible sheaf on $X$.

Since $\wt{\LL}$ is ample, there is a positive integer $m$ such that 
$\wt{\LL}^m$ is very ample, hence induces a closed immersion $\wt{\varphi}:
\wt{X}\ra \PP^N_{W_2(k)}$ with $\wt{\varphi}^*\OO_{\PP^N_{W_2(k)}}(1)=
\wt{\LL}^m$. Let $\iota:\PP^N_k\hookrightarrow \PP^N_{W_2(k)}$ be the closed 
immersion induced by $\iota:\spec k\hookrightarrow \spec W_2(k)$. 
Then it is easy to see that $X=\wt{X}\times_{\PP^N_{W_2(k)}}\PP^N_k$. 
Let $\varphi:X\ra \PP^N_k$ be the induced closed immersion.
\[
\xymatrix{
X \ar@{^{(}->}[r]^\varphi \ar@{^{(}->}[d]^\iota & \PP^N_k 
\ar[r] \ar@{^{(}->}[d]^\iota & \spec k \ar@{^{(}->}[d]^\iota \\
\wt{X} \ar@{^{(}->}[r]^{\wt{\varphi}} & \PP^N_{W_2(k)} \ar[r] & 
\spec W_2(k)
}
\]
Then we have $\varphi^*\OO_{\PP^N_k}(1)=\LL^m$, hence $\LL$ is ample. 
Taking $m$ sufficiently large, we may assume that $\LL^m$ is very ample and 
$H^1(X,\LL^m)=0$. Tensoring the exact sequence (\ref{es3}) by $\wt{\LL}^m$, 
we obtain an exact sequence:
\begin{eqnarray*}
0\ra \LL^m\stackrel{p}{\ra} \wt{\LL}^m\stackrel{r}{\ra} \LL^m\ra 0,
\end{eqnarray*}
which implies that $H^0(\wt{X},\wt{\LL}^m)\stackrel{r}{\ra} H^0(X,\LL^m)$ 
is surjective. 

Since $X$ is a smooth projective scheme over $k$, it follows from Bertini's 
theorem \cite[Theorem II.8.18]{ha77} that associated to a general section 
$s\in H^0(X,\LL^m)$, the divisor of zeros $D=\divisor_0(s)$ is smooth over 
$k$. Take a section $\wt{s}\in H^0(\wt{X},\wt{\LL}^m)$ with $r(\wt{s})=s$, 
then the divisor of zeros $\wt{D}=\divisor_0(\wt{s})$ is a lifting of $D$ 
over $W_2(k)$. By Lemma \ref{2.2}, $\wt{D}$ is smooth over $W_2(k)$.
\end{proof}

\begin{defn}\label{2.4}
Let $X$ be a smooth scheme over $k$. $X$ is said to be strongly liftable 
over $W_2(k)$, if there is a lifting $\wt{X}$ of $X$ over $W_2(k)$, such that 
for any prime divisor $D$ on $X$, $(X,D)$ has a lifting $(\wt{X},\wt{D})$ 
over $W_2(k)$ as in Definition \ref{2.1}, where $\wt{X}$ is fixed for all 
liftings $\wt{D}$.
\end{defn}

Let $X$ be a smooth scheme over $k$, $\wt{X}$ a lifting of $X$ over $W_2(k)$, 
$D$ a prime divisor on $X$ and $\LL_D=\OO_X(D)$ the associated invertible 
sheaf on $X$. Then there is an exact sequence of abelian sheaves:
\begin{eqnarray}
0\ra \OO_X\stackrel{q}{\ra} \OO^*_{\wt{X}}\stackrel{r}{\ra}\OO^*_X\ra 1, 
\label{es4}
\end{eqnarray}
where $q(x)=p(x)+1$ for $x\in\OO_X$, $p:\OO_X\ra p\cdot\OO_{\wt{X}}$ is 
the isomorphism (\ref{es2}) and $r$ is the reduction modulo $p$. The exact 
sequence (\ref{es4}) gives rise to an exact sequence of cohomology groups:
\begin{eqnarray}
H^1(\wt{X},\OO^*_{\wt{X}})\stackrel{r}{\ra} H^1(X,\OO^*_X)\ra 
H^2(X,\OO_X). \label{es5}
\end{eqnarray}
If $r:H^1(\wt{X},\OO^*_{\wt{X}})\ra H^1(X,\OO^*_X)$ is surjective, 
then $\LL_D$ has a lifting $\wt{\LL}_D$, i.e.\ $\wt{\LL}_D$ is an 
invertible sheaf on $\wt{X}$ with $\wt{\LL}_D|_X=\LL_D$. Tensoring the 
exact sequence (\ref{es3}) by $\wt{\LL}_D$, we have an exact sequence of 
$\OO_{\wt{X}}$-modules:
\begin{eqnarray*}
0\ra \LL_D\stackrel{p}{\ra} \wt{\LL}_D\stackrel{r}{\ra} \LL_D\ra 0,
\end{eqnarray*}
which gives rise to an exact sequence of cohomology groups:
\begin{eqnarray}
H^0(\wt{X},\wt{\LL}_D)\stackrel{r_D}{\ra} H^0(X,\LL_D)\ra H^1(X,\LL_D). 
\label{es7}
\end{eqnarray}

We recall here a sufficient condition for strong liftability 
\cite[Proposition 3.4]{xie10}.

\begin{prop}\label{2.5}
Let $X$ be a smooth scheme over $k$, and $\wt{X}$ a lifting of $X$ over 
$W_2(k)$. Then $X$ is strongly liftable if the following two conditions hold:
\begin{itemize}
\item[(i)] $r:H^1(\wt{X},\OO_{\wt{X}}^*)\ra H^1(X,\OO_X^*)$ is surjective;

\item[(ii)] For any prime divisor $D$ on $X$, there is a lifting 
$\wt{\LL}_D$ of $\LL_D=\OO_X(D)$ such that $r_D:H^0(\wt{X},\wt{\LL}_D)
\ra H^0(X,\LL_D)$ is surjective.
\end{itemize}
\end{prop}

\begin{lem}\label{2.6}
Let $X$ be a strongly liftable smooth scheme, and $\wt{X}$ a lifting 
of $X$ as in Definition \ref{2.4}. Then the natural map 
$r:H^1(\wt{X},\OO_{\wt{X}}^*)\ra H^1(X,\OO_X^*)$ is surjective.
\end{lem}

\begin{proof}
Let $D=\sum a_iD_i$ be an effective divisor on $X$. Then we can 
lift each distinct prime divisor $D_i$ to $\wt{D}_i$ on $\wt{X}$, 
hence $D$ has a lifting $\wt{D}=\sum a_i\wt{D}_i$, which is an effective 
Cartier divisor on $\wt{X}$. Let $\LL$ be an invertible sheaf on $X$. 
Assume $\LL=\OO_X(E_1-E_2)$, where $E_1$ and $E_2$ are effective 
divisors on $X$ without common irreducible components, then we can 
lift $E_i$ to $\wt{E}_i$ on $\wt{X}$. Define $\wt{\LL}=\OO_{\wt{X}}
(\wt{E}_1-\wt{E}_2)$. Then $\wt{\LL}$ is a lifting of $\LL$ on $\wt{X}$. 
Thus $r:H^1(\wt{X},\OO_{\wt{X}}^*)\ra H^1(X,\OO_X^*)$ is surjective.
\end{proof}

\begin{rem}\label{2.7}
Let $X$ be a strongly liftable smooth scheme, $\wt{X}$ a lifting of $X$ 
as in Definition \ref{2.4}, $\LL$ an invertible sheaf on $X$, and $\wt{\LL}$ 
a lifting of $\LL$ on $\wt{X}$. Then in general the natural map 
$r:H^0(\wt{X},\wt{\LL})\ra H^0(X,\LL)$ is not necessarily surjective. 
The following example is given by Illusie and Kato. Let $C$ be a smooth 
projective curve of genus $g\geq 1$. Then $C$ is strongly liftable by 
\cite[Theorem 1.3(i)]{xie10}. Let $\wt{C}$ be a lifting of $C$. 
By the exact sequence (\ref{es4}), we have an exact sequence 
of cohomology groups:
\begin{eqnarray*}
0\ra H^1(C,\OO_C)\ra H^1(\wt{C},\OO^*_{\wt{C}})\stackrel{r}{\ra} 
H^1(C,\OO^*_C).
\end{eqnarray*}
Since $\dim H^1(C,\OO_C)=g\geq 1$, there exists a non-trivial invertible 
sheaf $\wt{\LL}$ on $\wt{C}$ such that $\LL=\wt{\LL}|_C\cong\OO_C$. Then 
the section $s\in H^0(C,\LL)$ giving an isomorphism $\OO_C\ra\LL$ cannot 
be lifted to a section of $\wt{\LL}$. Thus $r:H^0(\wt{C},\wt{\LL})\ra H^0
(C,\LL)$ is not surjective.
\end{rem}

For convenience of the reader, we recall the Kawamata-Viehweg vanishing 
theorem for smooth projective varieties in positive characteristic under 
the lifting condition over $W_2(k)$ of certain log pairs, which has first 
been proved by Hara \cite{ha98}.

\begin{thm}[Kawamata-Viehweg vanishing in char.\ $p>0$]\label{2.8}
Let $X$ be a smooth projective variety of dimension $d$, 
$H$ an ample $\Q$-divisor on $X$, and $D$ a simple normal crossing 
divisor containing $\Supp(\langle H\rangle)$. 
Assume that $(X,D)$ admits a lifting over $W_2(k)$. Then 
\[ H^i(X,\Omega_X^j(\log D)(-\ulcorner H\urcorner))=0 \,\,\,\,
\hbox{holds for any} \,\,\,\, i+j<\inf(d,p). \]
In particular, $H^i(X,K_X+\ulcorner H\urcorner)=0$ holds 
for any $i>d-\inf(d,p)$.
\end{thm}

If $X$ is a smooth projective variety strongly liftable over 
$W_2(k)$, then the Kawamata-Viehweg vanishing theorem holds on $X$, 
which is a direct consequence of Definition \ref{2.4} and Theorem 
\ref{2.8}.

\section{Proofs of the main theorems}\label{S3}

In \cite{xie10}, it was proved that $\A^n_k$, $\PP^n_k$, smooth projective 
curves, smooth rational surfaces, and certian smooth complete intersections 
in $\PP^n_k$ are strongly liftable over $W_2(k)$. It follows easily from 
Proposition \ref{2.5} that smooth affine schemes are strongly liftable over 
$W_2(k)$. In this section, we shall give further examples of strongly liftable 
schemes and some applications to vanishing theorems.

\begin{thm}\label{3.2}
Any smooth toric variety is strongly liftable over $W_2(k)$.
\end{thm}

\begin{proof}
First of all, we recall some definitions and notations for toric varieties 
from \cite{fu93}. Let $N$ be a lattice of rank $n$, and $M$ the dual lattice 
of $N$. Let $\Delta$ be a fan consisting of strongly convex rational 
polyhedral cones in $N_\R$, and let $A$ be a ring. In general, we denote 
the toric variety associated to the fan $\Delta$ over the ground ring $A$ 
by $X(\Delta,A)$. More precisely, to each cone $\sigma$ in $\Delta$, there 
is an associated affine toric variety $U_{(\sigma,A)}=\spec A[\sigma^\vee
\cap M]$, and these $U_{(\sigma,A)}$ can be glued together to form the toric 
variety $X(\Delta,A)$ over $\spec A$. Note that almost all definitions, 
constructions and results for toric varieties are independent of the ground 
ring $A$, although everything is stated in \cite{fu93} over the complex 
number field $\C$.

Let $X=X(\Delta,k)$ be a smooth toric variety over $k$ associated to a fan 
$\Delta$. Let $\wt{X}=X(\Delta,W_2(k))$. Note that $U_{(\sigma,k)}=\spec 
k[\sigma^\vee\cap M]$, $U_{(\sigma,W_2(k))}=\spec W_2(k)[\sigma^\vee\cap M]$ 
is flat over $W_2(k)$ and $U_{(\sigma,W_2(k))}\times_{\spec W_2(k)}\spec k
=U_{(\sigma,k)}$, hence $\wt{X}$ is a lifting of $X$ over $W_2(k)$.

By \cite[Page 74, Corollary]{fu93}, we have $H^2(X,\OO_X)=0$, which implies 
that any invertible sheaf $\LL$ on $X$ has a lifting $\wt{\LL}$ on $\wt{X}$ 
by the exact sequence (\ref{es5}). In fact, we can prove the liftability of 
invertible sheaves in an explicit way. Let $\LL$ be an invertible sheaf on 
$X$. By \cite[Page 63, Proposition]{fu93}, we have an exact sequence:
\[
0\ra M\ra \Div_T(X)\ra \Pic(X)\ra 0.
\]
Therefore there exists a torus invariant divisor $D$ on $X$ such that 
$\LL=\OO_X(D)$. Assume that $\{ u(\sigma)\in M/M(\sigma) \}\in\projlim 
M/M(\sigma)$ determines the torus invariant divisor $D$. Then the same data 
$\{ u(\sigma)\in M/M(\sigma) \}$ also determines a torus invariant divisor 
$\wt{D}$ on $\wt{X}$ (we have only to change the base $k$ into $W_2(k)$). 
Thus the invertible sheaf $\LL=\OO_X(D)$ has a lifting $\wt{\LL}=
\OO_{\wt{X}}(\wt{D})$ on $\wt{X}$.

Let $E$ be a prime divisor on $X$, and $\LL=\OO_X(E)$ the associated 
invertible sheaf on $X$. Then we can take torus invariant divisors $D$ 
and $\wt{D}$ as above such that $\LL=\OO_X(D)$ and the invertible sheaf 
$\wt{\LL}=\OO_{\wt{X}}(\wt{D})$ lifts $\LL$. 
Let $v_i$ be the first lattice points in the edges of the cones 
of maximal dimension in $\Delta$, $D_i$ the corresponding orbit 
closures in $X$, and $\wt{D}_i$ the corresponding orbit closures in $\wt{X}$ 
($1\leq i\leq N$). Then the torus invariant divisors $D=\sum^N_{i=1}a_iD_i$ 
and $\wt{D}=\sum^N_{i=1}a_i\wt{D}_i$ determine a rational convex polyhedral 
$P_D$ in $M_\R$ defined by
\[
P_D=\{ u\in M_\R\,\,|\,\,\langle u,v_i\rangle\geq -a_i,\,\,1\leq i\leq N \}.
\]
By \cite[Page 66, Lemma]{fu93}, we have
\[
H^0(X,D)=\bigoplus_{u\in P_D\cap M}k\cdot\chi^u,\,\,
H^0(\wt{X},\wt{D})=\bigoplus_{u\in P_D\cap M}W_2(k)\cdot\chi^u.
\]
Thus the map $H^0(\wt{X},\wt{D})\stackrel{r}{\ra}H^0(X,D)$ induced by the 
natural surjection $W_2(k)\stackrel{r}{\ra}k$ is obviously surjective. 
Hence $r_E:H^0(\wt{X},\wt{\LL})\ra H^0(X,\LL)$ is surjective. By Proposition 
\ref{2.5}, $X$ is strongly liftable over $W_2(k)$.
\end{proof}

As a consequence, we have the Kawamata-Viehweg vanishing theorem on smooth 
projective toric varieties in positive characteristic for ample $\Q$-divisors 
which are not necessarily torus invariant, whereas Musta\c{t}\v{a} \cite{mu} 
and Fujino \cite{fu07} have proved certain more general versions of the 
Kawamata-Viehweg vanishing theorem on complete toric varieties in arbitrary 
characteristic for nef and big torus invariant $\Q$-divisors.

\begin{cor}\label{3.3}
Let $X$ be a smooth projective toric variety of dimension $d$, 
$H$ an ample $\Q$-divisor on $X$, and $D$ a simple normal crossing 
divisor containing $\Supp(\langle H\rangle)$. Then 
\[ H^i(X,\Omega_X^j(\log D)(-\ulcorner H\urcorner))=0 \,\,\,\,
\hbox{holds for any} \,\,\,\, i+j<\inf(d,p). \]
In particular, $H^i(X,K_X+\ulcorner H\urcorner)=0$ holds 
for any $i>d-\inf(d,p)$.
\end{cor}

\begin{proof}
It follows from Theorems \ref{2.8} and \ref{3.2}.
\end{proof}

\begin{cor}\label{3.4}
Let $X$ be a smooth toric variety of dimension $d$, $\pi:X\ra S$ a projective 
toric morphism onto a toric variety $S$, and $H$ a $\pi$-ample $\Q$-divisor 
on $X$ with $\Supp(\langle H\rangle)$ being simple normal crossing. Then 
$R^i\pi_*\OO_X(K_X+\ulcorner H\urcorner)=0$ holds for any $i>d-\inf(d,p)$.
\end{cor}

\begin{proof}
It follows from Corollary \ref{3.3} and a similar proof to that of 
\cite[Theorem 1-2-3]{kmm}.
\end{proof}

The following vanishing result \cite[Corollary 2.2.5]{kk} is useful, 
which holds in arbitrary characteristic.

\begin{lem}\label{3.5}
Let $f:Y\ra X$ be a proper birational morphism between normal surfaces 
with $Y$ smooth and with exceptional locus $E=\cup_{i=1}^s E_i$. 
Let $L$ be an integral divisor on $Y$, $0\leq b_1,\cdots,b_s<1$ 
rational numbers, and $N$ an $f$-nef $\Q$-divisor on $Y$. Assume 
$ L\equiv K_Y+\sum_{i=1}^s b_iE_i+N$. Then $R^1f_*\OO_Y(L)=0$ holds.
\end{lem}

The following lemma is a generalization of \cite[(1.2)]{el}, which holds 
in arbitrary characteristic.

\begin{lem}\label{3.6}
Let $X$ be a normal projective surface, $f:Y\ra X$ a resolution, and $H$ a 
$\Q$-Cartier $\Q$-divisor on $X$. If $H^1(Y,K_Y+\ulcorner f^*H\urcorner)=0$, 
then $H^1(X,K_X+\ulcorner H\urcorner)=0$ holds.
\end{lem}

\begin{proof}
Let $E=\cup_{i=1}^s E_i$ be the exceptional locus of $f$. 
Write $H=\ulcorner H\urcorner-\sum^r_{j=1}b_jD_j$, where $D_j$ are 
distinct prime divisors and $0<b_j<1$. Write 
$f^*D_j=f^{-1}_*D_j+\sum^s_{i=1}q_{ji}E_i$, where $f^{-1}_*D_j$ is 
the strict transform of $D_j$ and $q_{ij}\geq 0$. Thus we have
\begin{eqnarray}
f^*H&=&f^*\ulcorner H\urcorner-\sum^r_{j=1}b_jf^{-1}_*D_j
-\sum^s_{i=1}(\sum^r_{j=1}b_jq_{ji})E_i, \label{es8} \\
\ulcorner f^*H\urcorner&=&f^*\ulcorner H\urcorner
-\sum^s_{i=1}\big[\sum^r_{j=1}b_jq_{ji}\big]E_i. \label{es9}
\end{eqnarray}
By (\ref{es8}) and (\ref{es9}), we have
\begin{eqnarray*}
K_Y+\ulcorner f^*H\urcorner\equiv K_Y+\sum^s_{i=1}\big\langle\sum^r_{j=1}
b_jq_{ji}\big\rangle E_i+(f^*H+\sum^r_{j=1}b_jf^{-1}_*D_j),
\end{eqnarray*}
which implies $R^1f_*\OO_Y(K_Y+\ulcorner f^*H\urcorner)=0$ by Lemma \ref{3.5}.

By \cite[Propositions 5.75 and 5.77]{km}, the trace map $\Trace_{X/Y}:
f_*\OO_Y(K_Y)\ra \OO_X(K_X)$ is an injective homomorphism, furthermore, 
it is an isomorphism over the points where $f^{-1}$ is an isomorphism. 
It follows from (\ref{es9}) that we have the following injective homomorphism, 
which is an isomorphism over the points where $f^{-1}$ is an isomorphism:
\begin{eqnarray*}
f_*\OO_Y(K_Y+\ulcorner f^*H\urcorner)\ra \OO_X(K_X+\ulcorner H\urcorner),
\end{eqnarray*}
hence the quotient sheaf $\QQ$ of $\OO_X(K_X+\ulcorner H\urcorner)$ by 
$f_*\OO_Y(K_Y+\ulcorner f^*H\urcorner)$ is supported on a finite set of 
points. By the Leray spectral sequence, we have
\begin{eqnarray*}
H^1(X,f_*\OO_Y(K_Y+\ulcorner f^*H\urcorner))
=H^1(Y,K_Y+\ulcorner f^*H\urcorner)=0,
\end{eqnarray*}
which together with $H^1(X,\QQ)=0$ implies 
$H^1(X,K_X+\ulcorner H\urcorner)=0$.
\end{proof}

There is a generalization of \cite[Theorem 1.4]{xie10}, where no 
singularity assumption is made.

\begin{thm}\label{3.7}
Let $X$ be a normal projective surface and $H$ a nef and big $\Q$-divisor 
on $X$. If $X$ is birational to a strongly liftable smooth projective 
surface $Z$, then $H^1(X,K_X+\ulcorner H\urcorner)=0$ holds.
\end{thm}

\begin{proof}
Take a resolution $f:Y\ra X$ such that $\Supp(\{f^*H\})$ is simple normal 
crossing. Thus $(Y,\{f^*H\})$ is Kawamata log terminal (KLT, for short), 
and $K_Y+\ulcorner f^*H\urcorner-(K_Y+\{f^*H\})=f^*H$ is nef and big. 
By Kodaira's lemma, there exists an effective $\Q$-divisor $B$ on $Y$ 
such that $(Y,\{f^*H\}+B)$ is KLT and $K_Y+\ulcorner f^*H\urcorner-(K_Y+
\{f^*H\}+B)=f^*H-B$ is ample. By \cite[Theorem 4.2]{xie10}, we have 
$H^1(Y,K_Y+\ulcorner f^*H\urcorner)=0$, which implies 
$H^1(X,K_X+\ulcorner H\urcorner)=0$ by Lemma \ref{3.6}.
\end{proof}

There is also a generalization of \cite[Theorem 1.4]{xie09}.

\begin{cor}\label{3.8}
Let $X$ be a normal projective rational surface and $H$ a nef and big 
$\Q$-divisor on $X$. Then $H^1(X,K_X+\ulcorner H\urcorner)=0$ holds.
\end{cor}

\begin{proof}
It follows from \cite[Theorem 1.3]{xie10} and Theorem \ref{3.7}.
\end{proof}

By means of Lemma \ref{3.6}, we can give an alternative proof of 
\cite[Theorem 5.1]{sa}.

\begin{cor}\label{3.9}
Let $X$ be a normal projective surface over an algebraically closed 
field $k$ with $\ch(k)=0$, and $H$ a nef and big $\Q$-divisor on $X$. 
Then $H^1(X,K_X+\ulcorner H\urcorner)=0$ holds.
\end{cor}

\begin{proof}
Take a resolution $f:Y\ra X$ such that $\Supp(\langle f^*H\rangle)$ 
is simple normal crossing. Since $f^*H$ is nef and big, 
by \cite[Theorem 1-2-3]{kmm}, we have 
$H^1(Y,K_Y+\ulcorner f^*H\urcorner)=0$, which implies 
$H^1(X,K_X+\ulcorner H\urcorner)=0$ by Lemma \ref{3.6}.
\end{proof}

\section{Cyclic cover trick over $W_2(k)$}\label{S4}

The cyclic cover trick is a powerful technique, which is used widely 
in algebraic geometry. The general theory of cyclic covers over a field 
of arbitrary characteristic has already been given in \cite[\S 3]{ev}. 
In this section, we shall deduce the cyclic cover trick over $W_2(k)$ 
and use it to study the behavior of cyclic covers over strongly liftable 
schemes.

\begin{thm}\label{4.1}
Let $X$ be a smooth scheme, and $\LL$ an invertible sheaf on $X$. 
Let $N$ be a positive integer prime to $p$, and $D$ an effective divisor 
on $X$ with $\LL^N=\OO_X(D)$.
\begin{itemize}
\item[(i)] Let $\cA=\bigoplus^{N-1}_{i=0}\LL^{-i}(\big[\displaystyle\frac{iD}
{N}\big])$. Then $\cA$ is an $\OO_X$-algebra. Let $Y=\Spec\cA$. Then $Y$ is 
normal and the natural projection $\pi:Y\ra X$ is a finite surjective 
morphism of degree $N$, which is called the cyclic cover obtained by taking 
the $N$-th root out of $D$. Furthermore, if $\Sing(D_{\rm red})=\emptyset$, 
then $Y$ is smooth.

\item[(ii)] If $X$ has a lifting $\wt{X}$ over $W_2(k)$, $\LL$ has a lifting 
$\wt{\LL}$ on $\wt{X}$, and $D$ has a lifting $\wt{D}$ on $\wt{X}$ with 
$\wt{\LL}^N=\OO_{\wt{X}}(\wt{D})$, then $Y$ is liftable over $W_2(k)$. 
In fact, we have an induced cyclic cover $\wt{\pi}:\wt{Y}\ra \wt{X}$, 
such that $\wt{\pi}$ is a lifting of $\pi$ over $W_2(k)$. Furthermore, 
if $\Sing(D_{\rm red})=\emptyset$, then $\wt{Y}$ is smooth over $W_2(k)$, 
and $\wt{\pi}:\wt{Y}\ra \wt{X}$ has similar properties to those of 
$\pi:Y\ra X$, i.e.\ the statements of \cite[Claim 3.13 and Lemma 3.15]{ev} 
hold for $\wt{\pi}:\wt{Y}\ra \wt{X}$, where the phrase ``nonsingular" 
should be replaced by ``smooth over $W_2(k)$".
\end{itemize}
\end{thm}

\begin{proof}
(i) The construction and the properties of the cyclic cover $\pi:Y\ra X$ 
described as above have already been given in \cite[\S 3]{ev}.

(ii) Let $\wt{\cA}=\bigoplus^{N-1}_{i=0}\wt{\LL}^{-i}(\big[\displaystyle
\frac{i\wt{D}}{N}\big])$. By using the homomorphism $\wt{s}:\OO_{\wt{X}}
\ra \OO_{\wt{X}}(\wt{D})=\wt{\LL}^N$, we can prove that $\wt{\cA}$ 
is an $\OO_{\wt{X}}$-algebra. Let $\wt{Y}=\Spec\wt{\cA}$. Then 
$\wt{Y}\times_{\spec W_2(k)}\spec k=Y$. Since $\wt{\cA}$ is a locally 
free sheaf on $\wt{X}$ and $\wt{X}$ is smooth over $W_2(k)$ by Lemma 
\ref{2.2}, $\wt{Y}$ is flat over $W_2(k)$. Thus $\wt{Y}$ is a lifting 
of $Y$ over $W_2(k)$. Let $\wt{\pi}:\wt{Y}\ra\wt{X}$ be the natural 
projection. Then we have the following cartesian square:
\[
\xymatrix{
Y \ar[d]_{\pi} \ar@{^{(}->}[r]^\iota & \wt{Y} \ar[d]^{\wt{\pi}} \\
X \ar@{^{(}->}[r]^\iota & \wt{X}, 
}
\]
which implies that $\wt{\pi}$ is a lifting of $\pi$ over $W_2(k)$. 

If $\Sing(D_{\rm red})=\emptyset$, then $Y$ is smooth. By Lemma \ref{2.2}, 
$\wt{Y}$ is smooth over $W_2(k)$. In fact, the smoothness of $\wt{Y}$ 
can also be proved by a local calculation. Since $\Sing(D_{\rm red})=
\emptyset$, by Lemma \ref{2.2}, the irreducible components of $\wt{D}$ 
are disjoint and smooth over $W_2(k)$. Since the statements of 
\cite[Claim 3.13 and Lemma 3.15]{ev} (replacing ``nonsingular" by 
``smooth over $W_2(k)$") are local problems, we may assume that $X=\spec B$, 
$D=\alpha_1 D_1$, $\wt{X}=\spec\wt{B}$ and $\wt{D}=\alpha_1 \wt{D}_1$. 
We can factorize $\wt{\pi}:\wt{Y}\ra \wt{X}$ into two parts: one is an 
\'etale cover $\wt{\pi}_1:\wt{Z}\ra \wt{X}$ of degree $\gcd(N,\alpha_1)$, 
and the other is a ramified cover $\wt{\pi}_2:\wt{Y}\ra \wt{Z}$ of degree 
$N/\gcd(N,\alpha_1)$, hence the properties of $\wt{\pi}:\wt{Y}\ra \wt{X}$ 
follow from an almost identical local calculation to that of 
\cite[Claim 3.13]{ev}.
\end{proof}

First of all, we give a remark on Theorem \ref{4.1}.

\begin{rem}\label{4.2}
In Theorem \ref{4.1}(ii), if $D$ cannot be lifted over $W_2(k)$, then 
$Y$ is not necessarily liftable over $W_2(k)$. Such an example has been 
given in \cite{ra} and was generalized in \cite[Theorem 3.6]{xie07}. 
More precisely, given a Raynaud-Tango curve $C$, there are a $\PP^1$-bundle 
$f:X\ra C$ and an invertible sheaf $\LL=\OO_X(M)$ on $X$ such that 
$\LL^N=\OO_X(D)$ with $D=E+C'$, where $X,E,C',M$ are constructed as in 
\cite[Theorem 3.6]{xie07}. By \cite[Corollary 4.2]{xie09}, $X$ has a lifting 
$\wt{X}$ over $W_2(k)$. Since $H^2(X,\OO_X)=0$, by the exact sequence 
(\ref{es5}), $\LL$ has a lifting $\wt{\LL}$ on $\wt{X}$. However, the 
divisor $D$ cannot be lifted over $W_2(k)$ (see \cite[Corollary 1.10]{xie09} 
for the reason). Let $\pi:Y\ra X$ be the cyclic cover obtained by taking 
the $N$-th root out of $D$. Then $Y$ is a smooth projective surface and 
there exists a counterexample to the Kodaira vanishing theorem on $Y$ 
(see \cite[Theorem 3.6]{xie07} for the proof). Hence $Y$ is not liftable 
over $W_2(k)$ by \cite[Corollaire 2.8]{di}.
\end{rem}

\begin{cor}\label{4.3}
Let $X$ be a smooth projective variety, and $\LL$ an invertible sheaf on $X$. 
Let $N$ be a positive integer prime to $p$, and $D$ an effective divisor on 
$X$ with $\LL^N=\OO_X(D)$. Let $\pi:Y\ra X$ be the cyclic cover obtained by 
taking the $N$-th root out of $D$. If $X$ is strongly liftable over $W_2(k)$, 
$H^1(X,\LL^N)=0$ and $\Sing(D_{\rm red})=\emptyset$, then $Y$ is a smooth 
projective scheme which is liftable over $W_2(k)$.
\end{cor}

\begin{proof}
Since $X$ is strongly liftable, we can take a lifting $\wt{X}$ of 
$X$ as in Definition \ref{2.4}. By Lemma \ref{2.6}, $\LL$ has a lifting 
$\wt{\LL}$ on $\wt{X}$. Since $H^1(X,\LL^N)=0$, by the exact sequence 
(\ref{es7}), the natural map $r:H^0(\wt{X},\wt{\LL}^N)\ra H^0(X,\LL^N)$ 
is surjective. Thus the section $s\in H^0(X,\LL^N)$ corresponding to the 
effective divisor $D$ has a lifting $\wt{s}\in H^0(\wt{X},\wt{\LL}^N)$, 
which corresponds to an effective Cartier divisor $\wt{D}$ on $\wt{X}$ 
lifting $D$. Now, the conclusion follows from Theorem \ref{4.1}(ii).
\end{proof}

\begin{cor}\label{4.4}
Let $X$ be a smooth projective toric variety, and $\LL$ an invertible sheaf 
on $X$. Let $N$ be a positive integer prime to $p$, and $D$ an effective 
divisor on $X$ with $\LL^N=\OO_X(D)$ and $\Sing(D_{\rm red})=\emptyset$. 
Let $\pi:Y\ra X$ be the cyclic cover obtained by taking the $N$-th root 
out of $D$. Then $Y$ is a smooth projective scheme which is liftable over 
$W_2(k)$.
\end{cor}

\begin{proof}
Assume $X=X(\Delta,k)$. Then $\wt{X}=X(\Delta,W_2(k))$ is a lifting of 
$X$. By a similar argument to the proof of Theorem \ref{3.2}, we can 
take a lifting $\wt{\LL}$ of $\LL$ such that the natural map 
$r:H^0(\wt{X},\wt{\LL}^N)\ra H^0(X,\LL^N)$ is surjective. 
Thus the section $s\in H^0(X,\LL^N)$ corresponding to the 
effective divisor $D$ has a lifting $\wt{s}\in H^0(\wt{X},\wt{\LL}^N)$, 
which corresponds to an effective Cartier divisor $\wt{D}$ on $\wt{X}$ 
lifting $D$. Now, the conclusion follows from Theorem \ref{4.1}(ii).
\end{proof}

By using Corollary \ref{4.4}, we can construct a large class of liftable 
smooth projective varieties obtained by taking cyclic covers over toric 
varieties. In particular, there do exist many liftable smooth projective 
varieties of general type. We give the following easy examples to illustrate 
this idea.

\begin{ex}\label{4.5}
(i) Let $X=\PP(\OO\oplus\OO(-n))$ be the Hirzebruch surface over $\PP^1_k$, 
where $n\geq 0$ is an even number and $p\geq 3$. Then there are two natural 
sections $E_1,E_2$ of $f:X\ra \PP^1$ such that $E_2\sim E_1+nF$, $E_1^2=-n$, 
$E_2^2=n$ and $E_1\cap E_2=\emptyset$, where $F$ is the fiber of $f$. Let 
$\LL=\OO_X(E_1+\frac{n}{2}F)$. Then $\LL^2=\OO_X(E_1+E_2)$. Let $\pi:Y\ra X$ 
be the cyclic cover obtained by taking the square root out of $E_1+E_2$. Since 
$X$ is a toric variety and $\Sing(E_1+E_2)=\emptyset$, by Corollary \ref{4.4}, 
$Y$ is a liftable smooth projective surface.

(ii) Let $X=\PP^n_k$, $\LL=\OO_X(1)$ and $N$ a positive integer such that 
$n\geq 2$, $(N,p)=1$ and $N>n+2$. Let $H$ be a general element in the linear 
system of $\OO_X(N)$. Then $H$ is a smooth irreducible hypersurface of degree 
$N$ in $X$ with $\LL^N=\OO_X(H)$. Let $\pi:Y\ra X$ be the cyclic cover 
obtained by taking the $N$-th root out of $H$. Then by Corollary 4.4, 
$Y$ is a liftable smooth projective variety. By Hurwitz's formula, we have 
$K_Y=\pi^*(K_X+\frac{N-1}{N}H)$. Since the degree of $K_X+\frac{N-1}{N}H$ is 
$N-(n+2)>0$, $K_Y$ is an ample divisor on $Y$, hence $Y$ is of general type.
\end{ex}

Finally, we give a remark on Corollary \ref{4.3}.

\begin{rem}\label{4.6}
With the same notation and assumptions as in Corollary \ref{4.3}, 
although $Y$ is liftable over $W_2(k)$, in general, we cannot hope 
that $Y$ is strongly liftable over $W_2(k)$. The following local 
example given by Illusie shows that it is a quite subtle problem to 
determine whether $Y$ is strongly liftable or not.

Let $X=\A^2_k=\spec k[x,y]$, $D$ the line $x+1=0$ and $\pi:Y\ra X$ 
the cyclic cover obtained by taking the square root out of $D$. 
Let $\wt{X}=\A^2_{W_2(k)}=\spec W_2(k)[x,y]$, $\wt{D}$ the line $x+1=0$ 
and $\wt{\pi}:\wt{Y}\ra \wt{X}$ the cyclic cover obtained by taking 
the square root out of $\wt{D}$. It is easy to see that $\wt{\pi}:\wt{Y}\ra 
\wt{X}$ is a lifting of $\pi:Y\ra X$.

Let $E$ be the nodal cubic curve $y^2=x^2(x+1)$ in $X$. Then $\pi^*(E)$ 
consists of two irreducible components $E_1,E_2$, which are smooth and 
intersect transversally. More precisely, let $t^2=x+1$, then $E_1,E_2$ 
are defined by $y=\pm t(t^2-1)$ respectively. Thus both $E_1$ and $E_2$ 
are isomorphic to the normalization of $E$ and map onto $E$ through $\pi$. 
On the other hand, take a lifting $\wt{E}\subset \wt{X}$ of $E$, which is 
defined by $y^2=x(x+1)(x+p)$. Then $\wt{E}_{12}=\wt{\pi}^*(\wt{E})$ is an 
irreducible divisor on $\wt{Y}$ defined by $y^2=(t^2-1)t^2(t^2+p-1)$. Since 
$\wt{E}_{12}$ is not relatively simple normal crossing over $W_2(k)$, by 
Lemma \ref{2.2}, it can never be a lifting of $E_1$ or $E_2$ or $E_1+E_2$.
\[
\xymatrix{
E_1+E_2 \ar[d]_{\pi} \ar@{.>}[r]|\times & \wt{E}_{12} \ar[d]^{\wt{\pi}} \\
E \ar@{^{(}->}[r] & \wt{E}
}
\]

For example, if we would like to lift $E_1\subset Y$, then first we need to 
lift $E=\pi_*(E_1)\subset X$. If $E$ has a unique lifting, e.g.\ $\wt{E}$ 
defined as above, then $\wt{E}_{12}=\wt{\pi}^*(\wt{E})$ is not a lifting of 
$E_1$. Hence $E_1$ cannot be lifted over $W_2(k)$, which shows that the 
strong liftability of $X$ is not sufficient to ensure that the cyclic cover 
$Y$ is also strongly liftable.
\end{rem}

\section{Some corrections}\label{S5}

In this section, we will give some corrections to the mistakes in 
\cite{xie10}. First of all, the necessity of \cite[Proposition 3.4]{xie10} 
does not hold in general (see Remark \ref{2.7} for the reason), and it 
has already been changed into Proposition \ref{2.5}.

\begin{rem}\label{5.1}
If $C$ is a Tango curve, then there is a decomposable locally free sheaf 
$\FF$ of rank two on $C$ such that $Z=\PP(\FF)$ is not strongly liftable 
over $W_2(k)$. We use the same notation and construction as in 
\cite[Theorem 3.1]{xie07}. Let $\FF=\OO_C\oplus\LL$. Then $\FF\cong\EE$ on 
a dense open subset of $C$, hence $Z=\PP(\FF)$ is birational to $X=\PP(\EE)$. 
Assume the contrary that $Z$ is strongly liftable. Then by 
\cite[Theorem 4.2]{xie10}, the Kawamata-Viehweg vanishing theorem 
should hold on $X$, which is a contradiction. This example shows that 
\cite[Proposition 3.8]{xie10} does not hold in general. In addition, 
the conclusion in \cite[Remark 3.9]{xie10} does not hold yet in general. 
Thus we change \cite[Proposition 3.8]{xie10} into the following statement.
\end{rem}

\begin{prop}\label{5.2}
Let $C=\PP^1_k$, $\GG$ an invertible sheaf on $C$, and $\EE=\OO_C\oplus\GG$. 
Let $X=\PP(\EE)$ and $f:X\ra C$ the natural projection. Then $X$ is strongly 
liftable over $W_2(k)$.
\end{prop}

\begin{proof}
We proceed the same argument as in the proof of \cite[Proposition 3.8]{xie10}. 
Finally, we have only to show that each $\pi^i_D:H^0(\wt{C},\wt{\GG}^i
\otimes\wt{\HH})\ra H^0(C,\GG^i\otimes\HH)$ is surjective. 
If $\deg(\GG^i\otimes\HH)<0$, then $H^0(C,\GG^i\otimes\HH)=0$, hence $\pi^i_D$ 
is surjective. If $\deg(\GG^i\otimes\HH)\geq 0$, then $H^1(C,\GG^i\otimes\HH)
=0$ by Serre's duality, hence by the exact sequence (\ref{es7}), $\pi^i_D:H^0
(\wt{C},\wt{\GG}^i\otimes\wt{\HH})\ra H^0(C,\GG^i\otimes\HH)$ is surjective.

In fact, we have an alternative proof. Since $X$ is a Hirzebruch surface, 
$X$ is toric, hence is strongly liftable by Theorem \ref{3.2}.
\end{proof}

Finally, we change \cite[Theorem 1.3]{xie10} into the following statement.

\begin{thm}\label{5.3}
The following schemes are strongly liftable:
\begin{itemize}
\item[(i)] $\A^n_k$, $\PP^n_k$ and a smooth projective curve;

\item[(ii)] a smooth projective variety of Picard number 1 
which is a complete intersection in $\PP^n_k$;

\item[(iii)] a smooth projective rational surface.
\end{itemize}
\end{thm}

\begin{proof}
Both (i) and (ii) have already been proved in \cite{xie10}. (iii) follows 
from \cite[Proposition 2.6]{xie10} and Proposition \ref{5.2}.
\end{proof}

\textsc{School of Mathematical Sciences, Fudan University, 
Shanghai 200433, China}

\textit{E-mail address}: \texttt{qhxie@fudan.edu.cn}

\end{document}